\documentclass[leqno]{article}
\usepackage{amsmath}
\usepackage{amssymb}
\usepackage{amsfonts}
\usepackage{mathbbol}
\usepackage{amsthm}

\makeatletter
 \@addtoreset{equation}{section}
 
 \makeatother

\theoremstyle{plain}
\newtheorem{thm}{Theorem}[section]

\newtheorem{lem}[thm]{Lemma}

\newtheorem{rem}[thm]{Remark}

\newtheorem{theoremIntro}{Theorem}

\title{An upper bound for the critical probability on the Cartesian product graph
		of a regular tree and a line }
\author{Kohei Yamamoto}
\date{Tohoku university}

\begin{document}
\maketitle

\begin{abstract}
	We study Bernoulli bond percolation on the Cartesian product graph of
	a regular tree and a line.
	We give an upper bound for the critical probability,
	which improves previous upper bound.
	We use a method  which is similar to Golton-Watson process.
	Our result leads that there exists a non-empty phase in which 
	there are infinitely many infinite clusters 
	when a degree of a tree is $4$.

\end{abstract}
\renewcommand{\thefootnote}{\fnsymbol{footnote}}

\section{Introduction}
\label{se:Intro}
	Let $G=(V,E,o)$ be a rooted, connected, locally finite, and infinite graph,
	where $V$ is the set of vertices, $E$ is the set of edges, 
	and $o$ is a special vertex called a root.
	In Bernoulli bond percolation, 
	each edge will be open with probability $p$,
	and closed with probability $1-p$ independently, 
	where $p \in [0,1]$ is a fixed parameter.
	Let $\Omega=2^E$ be the set of samples,
	where $\omega \subset E$ is the set of all open edges.
	Each $\omega \in \Omega$ is regarded as a subgraph of $G$ 
	consisting of all open edges.
	The connected components of $\omega$ are referred to as clusters.
	Let $p_c=p_c(G)$ be the critical probability for Bernoulli bond percolation 
	on $G$, that is,
	\[
	p_c= \inf \left\{ p\in[0,1] \mid 
	\text{there exists an infinite cluster almost surely} \right\},
	\]
	and let $p_u=p_u(G)$ be the uniqueness threshold for Bernoulli bond percolation 
	on $G$, that is,
	\[
	p_u= \inf \left\{ p\in[0,1] \mid 
	\text{there exists an unique infinite cluster almost surely} \right\}.
	\]
	One of the most popular graphs in the theory of percolation is 
	the Euclidean lattice $\mathbb{Z}^d$.
	In 1980 Kesten\cite{Kesten} proved that $p_c=1/2$ in the case of two dimensions.
	But in the case of three dimensions or more,
	as a numerical value, the critical probability is not quite clear. 
	Regarding the uniqueness threshold of the Euclidean lattice,
	in 1987 Aizenman, Kesten, and Newman\cite{Aizenman} proved that there exists
	at most one infinite cluster almost surely for all $d \ge 1$,
	that is, they showed that  $p_c=p_u$ for all $d \ge 1$. 
	The Cartesian product graph of a $d$-regular tree and a line $T_d \Box \mathbb{Z}$
	was presented as a first example of a graph with $p_c < p_u<1$
	by Grimmett and Newman\cite{Grimmett2} in 1990.
	They showed that $p_c<p_u$ holds when $d$ is sufficiently large.
	After this article had appeared,
	percolation on $T_d \Box \mathbb{Z}$ has been a popular topic.
	However, the critical probability of $T_d \Box \mathbb{Z}$ 
	is, as a value, also not quite clear.
	In recent years, Lyons and Peres\cite{Lyons1} 
	gave the following upper bound of $p_c$ and lower bound of $p_u$.
	\begin{thm}[\cite{Lyons1}]
	For all $d \ge 3$, we have 
	\begin{align}
	\label{ineq:p_c}
	p_c(T_d \Box \mathbb{Z}) &\le \frac{d-\sqrt{d^2-4}}{2},\\
	\label{ineq:p_u}
	p_u(T_d \Box \mathbb{Z}) &\ge \left(\sqrt{d-1}+1+\sqrt{2\sqrt{d-1}-1}\right)^{-1}.
	\end{align}
	\end{thm}
	From this result, $p_c <p_u$ holds for all $d \ge 5$.
	The main result of this paper is to give a new upper bound 
	which improves the inequality (\ref{ineq:p_c}).
	\begin{theoremIntro}
	\label{ThmA}
	Let $d \ge 3$. Then we have
	\[
	p_c(T_d \Box \mathbb{Z}) \le \frac{1}{d}.
	\]
	\end{theoremIntro}
	In the case of $d=4$, 
	we further improve this upper bound.
	\begin{theoremIntro}
	Let $d=4$. Then we have
	\begin{equation}
	\label{ineq:d=4}
	p_c< 0.225.
	\end{equation}
	\end{theoremIntro}
	Then using the inequality (\ref{ineq:p_u}) and (\ref{ineq:d=4}),
	we have $p_c<0.225<0.232<p_u$ when $d=4$.
	\begin{rem}
	A preprint of this paper first appeared in June 2017.
	At that time, it was not known whether $p_c < p_u$ holds when $d=3$.
	In November 2017, Hutchcroft\cite{Hutchcroft} showed that 
	$p_c < p_u$ holds for all $d \ge 3$.
	The method of Hutchcroft is different from ours.
	\end{rem}

\section{Probability generating function}
\label{se:PGF}
	The critical probability of $T_d$ is found by only consider Galton-Watson process.
	In this process, we can know whether the tree is infinite or not 
	by only consider the first step.
	Lyons and Peres used the natural projection from $T_d \Box \mathbb{Z}$
	to $T_d$, and focus on the first step.
	They gave an upper bound by using the method like Galton-Watson process.
	We also use the projection and essentially the same method.
	But our strategy is to consider each step not just the first step 
	to get a better estimate.

	In general, if $H$ is a subgraph of $G$ containing $o$,
	then we have $p_c(G) \le p_c(H)$.
	There exists a $(d-1)$-ary tree as a subgraph of a $d$-regular tree,
	where $(d-1)$-ary tree is a tree such that $\deg v =d$ except the root
	and $\deg o=d-1$.
	Then we can find an upper bound of $p_c(T_d \Box \mathbb{Z})$
	by estimating the critical probability of the Cartesian product graph of 
	a $(d-1)$-ary tree and a line.
	Therefore, we may assume that $T_d$ is a $(d-1)$-ary tree in the following.
	We denote the probability measure associated with this process by  
	$\mathbb{P}_p$
	or $\mathbb{P}_p^G$
	and the expectation operator by $\mathbb{E}_p$ or $\mathbb{E}_p^G$.
	The definition of $p_c$ can be rewritten using $\mathbb{P}_p$.
	Let $\mathcal{C}$ be a cluster containig $o$
	and $|\mathcal{C}|$ be a order of the vertex set of $\mathcal{C}$. 
	Then we can rewrite
	\[
	p_c= \inf \left\{ p\in[0,1] \mid 
	\mathbb{P}_p(|\mathcal{C}|=\infty)>0 \right\}.
	\]
	In Bernoulli percolation on $T_d$,
	we can show that $p_c=1/(d-1)$ by Galton-Watson process.
	In Galton-Watson process, let $X_n$ be the number of vertices such that 
	it has distance $n$ from the root.
	It is well known that following equation holds.
\begin{equation}
	\label{equ:GW}
	\mathbb{E}[X_n]=\mathbb{E}[X_1]^n.
\end{equation}
	Therefore if $\mathbb{E}[X_1] >1$, then the probability that the tree is infinite, 
	is positive.
	Let $B_n$ be a subgraph of $T_d$ defined by an $n$-ball centered at the root.
	Because we can decompose $T_d$ into several pieces each of which is
	isomorphic to $B_n$,
	the equation \eqref{equ:GW} holds.
	The graph $T_d \Box \mathbb{Z}$ has a similar structure.
	To explain it,
	let $H_n$ be a subgraph of $T_d \Box \mathbb{Z}$ defined by
\begin{align*}
	V(H_n)&=\left\{ (v, k) \mid d_{T_d}(v,0_{T_d}) \le n, k\in \mathbb{Z} \right\},\\
	E(H_n)&=\left\{ \{ (v,k),(u,k)\}  \mid 
			(v,k),(u,k) \in V(H_n), \{v,u\} \in E(T_d), k\in \mathbb{Z} \right\} \\
			& \quad \cup
			\left\{ \{ (v,k), (v,k+1)\} \mid d_{T_d}(v,0_{T_d}) \le n-1, k \in \mathbb{Z} \right\}, 
\end{align*}
	where $d_{T_d}$ is the graph metric on $T_d$.
	Then $T_d \Box \mathbb{Z}$ is clearly decomposed into infinitely many 
	copies of $H_n$.
	Let $\pi$ be the natural projection from $T_d \Box \mathbb{Z}$ to $T_d$.
	In percolation on $H_n$, 
	we define a random variable $X_n(\omega)$ by the following formula.
\begin{equation}
	X_n(\omega)= \# \left\{ v \in T_d \mid  d_{T_d}(v,0_{T_d}) = n, 
	o \leftrightarrow \pi^{-1}(v) \text{ on } H_n\right\},
\end{equation} 
	where the notation $A \leftrightarrow B$ means that 
	there exists an open path between $A$ and $B$.
	Similar to Galton-Watson process, we have the following lemma.
	
\begin{lem}
	\label{lem:PGF}
	For any $p$, 
	if there exists $n$ such that $\mathbb{E}_p[X_n]>1$, 
	then $\mathbb{P}_p( |\mathcal{C}|=\infty)>0$.
\end{lem}

\begin{proof}
	Let $s \in [0,1]$ be a parameter. The probability generating function is defined by
	\[
	f_n(s)=\mathbb{E}_p[s^{X_n}]=\sum_{k=0}^{(d-1)^n} \mathbb{P}_p(X_n=k)s^k.
	\]
	In Galton-Watson process, the equation $f_{nm}(s)=f_n(s)^{(m)}$ holds,
	where $f_n^{(m)}$ is the $m$ times composite function of $f_n$.
	Similarly, we will show the following inequality.
	\begin{equation}
		\label{ine:PGF}
		f_{nm}(s) \le f_n(s)^{(m)}.
	\end{equation}
	We compute $f_{nm}(s)$ by rewriting probabilities involving $X_{nm}$ 
	in terms of $X_{n(m-1)}$ as follows.
	\begin{align*}
		f_{nm}(s)
		&= \sum_{k \ge 0} \mathbb{P}_p(X_{nm}=k)s^k\\
		&= \mathbb{P}_p(X_{nm} \ge 0) 
			-\sum_{k \ge1} \mathbb{P}_p(X_{nm} \ge k)s^{k-1}(1-s)\\
		&=1-\sum_{k \ge1} \sum_{l \ge 1} \mathbb{P}_p(X_{nm} \ge k|X_{n(m-1)}=l)
			\cdot \mathbb{P}_p(X_{n(m-1)}=l)s^{k-1}(1-s).
	\end{align*}
	If $(X_{n(m-1)}=l)$ occurs,
	there exists $l$ vertices $v_1, \ldots ,v_l$ on $T_d$ 
	such that $d_{T_d} (v_i, o_{T_d})=n(m-1)$,
	and there exists an open path from $o$ to $\pi^{-1}(v_i)$ on $H_{n(m-1)}$.
	Therefore, for each $i$, there exists at least one vertex on $\pi^{-1}(v_i)$ which is 
	connected with the root by an open path.
	We regard these vertices as new roots.
	That is, we consider percolation on $H_{n(m-1)}$ first, 
	then we consider percolation on $H_{nm} \setminus H_{n(m-1)}$ next.
	Since $H_{nm} \setminus H_{n(m-1)}$ is a union of $(d-1)^{n(m-1)}$ pieces 
	each of which is isomorphic to $H_n$, 
	we can estimate a lower bound of $\mathbb{P}_p(X_{nm} \ge k|X_{n(m-1)}=l)$.
	\[
	\mathbb{P}_p(X_{nm} \ge k|X_{n(m-1)}=l) \ge 
	\sum_{\substack{j_1, \ldots , j_l \ge 0 \\j_1+\cdots+j_l \ge k}} 
	\left( \prod_{u=1}^l \mathbb{P}_p(X_n=j_u) \right).
	\]
	Therefore, we have
	\begin{align*}
		f_{nm}(s)
		& \le 1- \sum_{l \ge 1}  \sum_{j_1, \ldots , j_l \ge 0} 
		\left(\prod_{u=1}^l \mathbb{P}_p(X_n=j_u) \right) \cdot 
		\mathbb{P}_p(X_{n(m-1)}=l) \sum_{k=1}^{j_1+\cdots+j_l} s^{k-1}(1-s)\\
		&= 1- \sum_{l \ge 1}  \mathbb{P}(X_{n(m-1)}=l) 
		 \sum_{j_1, \ldots , j_l \ge 0} 
		\left(\prod_{u=1}^l \mathbb{P}_p(X_n=j_u) \right)\\
		&\qquad +\sum_{l \ge 1}  \mathbb{P}(X_{n(m-1)}=l) 
		\sum_{j_1, \ldots , j_l \ge 0} 
		\left(\prod_{u=1}^l \mathbb{P}_p(X_n=j_u) s^{j_u} \right)\\
		&= \mathbb{P}_p(X_{n(m-1)}=0) + \sum_{l \ge 1}  \mathbb{P}(X_{n(m-1)}=l) 
		\left( \mathbb{E}_p[s^{X_n}] \right)^l \\
		&= \mathbb{E}_p[f_n(s)^{X_{n(m-1)}}] \le \cdots \le f_n^{(m)}(s),
	\end{align*}
	which completes the proof of the inequality (\ref{ine:PGF}).
	By the definition of the probability generating function, 
	if $\mathbb{E}[X_n]>1$, 
	then we have $\displaystyle \lim_{m \to \infty} f_n^{(m)} (0) <1$.
	Using the inequality (\ref{ine:PGF}), 
	we have $\displaystyle \lim_{n \to \infty} f_n (0) <1$.
	Hence we have $\mathbb{P}_p( |\mathcal{C}|=\infty)>0$.
	\end{proof}
	Let $v_n$ be a vertex in $T_d$ such that $d_{T_d} (v_n, o_{T_d})=n$.
	We have
\begin{equation}
	\mathbb{E}[X_n]=(d-1)^n \mathbb{P}_p^{H_n} (o \leftrightarrow \pi^{-1}(v_n)).
\end{equation}
	Using this equation and Lemma \ref{lem:PGF}, 
	if $\displaystyle \limsup_{n \to \infty} 
	\mathbb{P}_p^{H_n} (o \leftrightarrow \pi^{-1}(v_n))^{1/n} > 1/(d-1)$,
	then we have $\mathbb{P}_p( |\mathcal{C}|=\infty)>0$.
	It is difficult to estimate $\displaystyle \limsup_{n \to \infty} 
	\mathbb{P}_p^{H_n} (o \leftrightarrow \pi^{-1}(v_n))^{1/n}$ exactly.
	Therefore, we take a subgraph $L_{n-1} \Box \mathbb{Z} \subset H_n$,
	where $L_n$ is a segment of length $n$ in $T_d$ emanating from $o_{T_d}$.
	Then we define a function $\alpha:[0,1] \to [0,1]$ by
	\begin{equation}
	\label{def:alpha}
	\alpha(p)=\limsup_{n \to \infty} 
	\mathbb{P}_p^{L_n \Box \mathbb{Z}} (o \leftrightarrow \pi^{-1}(v_n))^{1/n}.
	\end{equation} 
	Because $L_{n-1} \Box \mathbb{Z} \subset H_n$,
	if $\alpha(p) > 1/(d-1)$, then $\displaystyle \limsup_{n \to \infty} 
	\mathbb{P}_p^{H_n} (o \leftrightarrow \pi^{-1}(v_n))^{1/n}  \ge \alpha(p) > 1/(d-1)$.
	Hence, we have the following lemma.

\begin{lem}
\label{lem:limit}
	Let $p_0= \displaystyle \inf \left\{p\in [0,1] \mid \alpha(p)>1/(d-1) \right\}$,
	then we have $p_c(T_d \Box \mathbb{Z}) \le p_0$.
\end{lem}

\section{Lower bound of $\alpha(p)$}
\label{se:LB}
	We have defined a function $\alpha(p)$ in (\ref{def:alpha}) which is useful to give,
	as in Lemma \ref{lem:limit}, an upper bound for $p_c(T_d \Box \mathbb{Z})$.
	However, it is still difficult to handle.
	Thus, we shall prepare another lower bound which depends on both $p$ and $n$.
	Let $L_{\infty}=\mathbb{Z}_{\ge 0}$ be a ray.
	Let $H$ be a subgraph of $L_{\infty} \Box \mathbb{Z}$ defined by
\begin{align*}
	V(H)&=\left\{ (n, k) \mid n\in \{0,1\}, k\in \mathbb{Z} \right\},\\
	E(H)&= \left\{ \{ (0,k), (0,k+1)\} \mid  k \in \mathbb{Z} \right\}\\
			& \quad \cup
			\left\{ \{ (0,k),(1,k)\} \mid  k\in \mathbb{Z} \right\}.
\end{align*}
	$L_{\infty} \Box \mathbb{Z}$ is decomposed into infinitely many pieces
	each of which is isomorphic to $H$.
	We denote this decomposition as $L_{\infty} \Box \mathbb{Z}=\cup_{i=1}^{\infty}H_i$
	where $H_i$ is a copy of $H$.
	We set $G_n= \sqcup_{i=1}^n H_i$.
	Then we have
	\[
	\alpha(p)=\limsup_{n \to \infty} 
	\mathbb{P}_p^{G_n} (o \leftrightarrow \pi^{-1}(n))^{1/n}.
	\]
	We will make a lower bound of $\alpha(p)$.
	We define the sequence of numbers $\{ \alpha_p(n) \}$ by
	\[
	\alpha_p(n)=
	\sum_{\substack{m_i \ge 0 \\ 1 \le i \le n}}
	\sum_{\substack{ l_i =1   \\ 1 \le i \le n}}^{m_i +l_{i-1}}
	p^{\sum_{j=1}^n (m_j+l_j)} (1-p)^{2n+l_0 -l_n + \sum_{j=1}^n m_j}
	\left( \prod_{j=1}^n (m_j+1)\binom{m_j +l_{j-1}}{l_j} \right),
	\]
	where $l_0=1$.
	\begin{lem}
	\label{lem:LB}
	For all $n \ge 1$, we have
	\begin{equation}
	\mathbb{P}_p^{G_n}(o \leftrightarrow \pi^{-1}(n))
	\ge \alpha_p(n).
	\end{equation}
	\end{lem}
	\begin{proof}
	When $n=1$,
	$\mathbb{P}_p^{H_1}(o \leftrightarrow \pi^{-1}(1))$ can be computed exactly as follows.
	First, we consider percolation on $\pi^{-1}(0)$,
	which is isomorphic to $\mathbb{Z}$.
	Let $m$ be a nonnegative integers. Then we have
	\[
	\mathbb{P}_p^{\mathbb{Z}}(|\mathcal{C}|=m+1)=(m+1) p^m(1-p)^2.
	\]
	Second, we consider percolation on the remaining edges,
	that means on $\left\{ \{ (0,k),(1,k)\} \mid  k\in \mathbb{Z} \right\}$.
	We are now thinking on the case where the event $|\mathcal{C}|=m+1$ occurs.
	Thus, we only consider percolation on 
	$\left\{ \{ (0,k),(1,k)\} \mid  (0,k) \in \mathcal{C} \right\}$.
	If at least one of the $m+1$ edges is open,
	then $(o \leftrightarrow \pi^{-1}(1))$ occurs.
	So, we have 
	\begin{align}
	\mathbb{P}_p^{H_1}(o \leftrightarrow \pi^{-1}(1))
	&=\sum_{m \ge 0} (m+1) p^m(1-p)^2 (1-(1-p)^{m+1}) \notag \\
	\label{eq:H_1}
	&=\sum_{m \ge 0} (m+1) p^m(1-p)^2 \sum_{l=1}^{m+1} \binom{m+1}{l}p^l(1-p)^{m+1-l}.
	\end{align}
	Next, we would like to show general case.
	If $(o \leftrightarrow \pi^{-1}(1))$ occurs,
	then there exists a non-empty subset $A_1 \subset \pi^{-1}(1)$
	such that $o \leftrightarrow v$ on $H_1$ for all $v \in A_1$.
	We would like to know the probability 
	$\mathbb{P}_p^{H_2}(A_1 \leftrightarrow \pi^{-1}(2))$ with $|A_1|=l_1$.
	It depends on a configuration of $A_1$,
	but we can obtain a lower bound which does not depend on  a configuration of $A_1$.
	Since $\pi^{-1}(1)$ and $\mathbb{Z}$ are isomorphic,
	we replace $\pi^{-1}(1)$ with $\mathbb{Z}$.
	The case where $l_1=1$ is explain as above.
	So, we assume $l_1 \ge 2$. 
	Let $A_1=\{v_1, \ldots , v_{l_1}\}$ such that $v_1<v_2< \cdots <v_{l_1}$.
	First, we consider percolation on $\mathbb{Z} \setminus [v_1, v_{l_1}]$,
	and all edges of $[v_1, v_{l_1}]$ are assumed to be closed.
	We divide computation into several cases according to the size of the cluster 
	containing $A_1$.
	In other words, the cluster containing $A_1$ is
	$\mathcal{C}^{\prime}=\bigcup_{i=2}^{l-1} \{v_i\} 
	\cup \mathcal{C}(v_1) \cup \mathcal{C}(v_{l_1})$,
	where $\mathcal{C}(v_i)$ is the intersection of the cluster containing $v_i$
	and $\mathbb{Z} \setminus [v_1, v_{l_1}]$.
	Then we have
	\[
	\mathbb{P}_p^{\mathbb{Z}}
	(|\mathcal{C}^{\prime}|=m_2+l_1)=(m_2+1) p^{m_2}(1-p)^2.
	\]
	Second, we consider percolation on the remining edges.
	We are now thinking about the case where the event 
	$(|\mathcal{C}^{\prime}|=m_2+l_1)$ occurs.
	Thus, we only consider percolation on 
	$\left\{ \{ (0,k),(1,k)\} \mid  (0,k) \in \mathcal{C}^{\prime} \right\}$.
	If at least one of the $m_2+l_1$ edges is open,
	then $(A_1 \leftrightarrow \pi^{-1}(2))$ occurs.
	So, we have 
	\begin{align}
	\mathbb{P}_p^{H}(A_1 \leftrightarrow \pi^{-1}(2))
		& \ge \sum_{m_2 \ge 0} (m_2+1) p^m_2(1-p)^2 (1-(1-p)^{m_2+l_1}) \notag \\
	\label{ine:H_2}
		&= \sum_{m_2 \ge 0}  (m_2+1) p^m_2(1-p)^2
			\sum_{l_2=1}^{m_2+l_1} \binom{m_2+l_1}{l_2} 
			p^{l_2} (1-p)^{m_2+l_1-l_2}.
	\end{align}
	If $(o \leftrightarrow A_1)$ on $H_1$
	and $(A_1 \leftrightarrow \pi^{-1}(2))$ on $H_2$ occur,
	then $(o \leftrightarrow \pi^{-1}(2))$ on $G_2$ occur.
	Therefore, using \eqref{eq:H_1} and \eqref{ine:H_2}, we have
	\[
	\mathbb{P}_p^{G_2}(o \leftrightarrow \pi^{-1}(2))
	\ge \sum_{\substack{m_i \ge 0 \\ 1 \le i \le 2}}
	\sum_{\substack{ l_i =1   \\ 1 \le i \le 2}}^{m_i +l_{i-1}}
	p^{\sum_{j=1}^2 (m_j+l_j)} (1-p)^{2 \cdot 2+l_0 -l_2+\sum_{j=1}^2 m_j}
	\left( \prod_{j=1}^2 (m_j+1)\binom{m_j +l_{j-1}}{l_j} \right),
	\]
	where $l_0=1$. 
	If $(A_1 \leftrightarrow \pi^{-1}(2))$ occurs,
	then there exists a non-empty subset $A_2 \subset \pi^{-1}(2)$
	such that $A_1 \leftrightarrow v$ on $H_2$ for all $v \in A_2$.
	Repeating this process,
	If there exists non-empty subset $A_i \subset \pi^{-1}(i)$ 
	and $(A_{i-1} \leftrightarrow A_i)$ on $H_i$ occurs for $1 \le i \le n$
	where $A_0=o$,
	then $(o \leftrightarrow \pi^{-1}(n))$ on $G_n$ occurs.
	It complets the proof.
	\end{proof}
	By Lemma \ref{lem:LB}, we have
	\begin{equation}
	\alpha(p) \ge \limsup_{n \rightarrow \infty} \alpha_p(n)^{1/n}.
	\end{equation}

\section{Generating function and radius of convergence}
\label{se:GF}
	Since it is not quite easy to handle $\alpha_p(n)$, 
	we introduce another sequence of numbers 
	which is easier to handle than $\alpha_p(n)$.
	The sequence of numbers $\beta_p(n)$ is defined by
	\[
	\beta_p(n)
	=\sum_{\substack{m_i \ge 0 \\ 1 \le i \le n}}
	\sum_{\substack{ l_i =1   \\ 1 \le i \le n}}^{m_i +l_{i-1}}
	p^{\sum_{j=1}^n (m_j+l_j)} (1-p)^{2n+l_0+\sum_{j=1}^n m_j}
	\left( \prod_{j=1}^n (m_j+1)\binom{m_j +l_{j-1}}{l_j} \right),
	\]
	where $l_0=1$.
	Since $\alpha_p(n) \ge \beta_p(n)$ for all $n$,
	we have  $\displaystyle \limsup_{n \rightarrow \infty} \alpha_p(n)^{1/n} \ge 
	\limsup_{n \rightarrow \infty} \beta_p(n)^{1/n}$.
	We know $\displaystyle \limsup_{n \rightarrow \infty} \beta_p(n)^{1/n}$ 
	equals the inverse of the radius of convergence of the generating function	
	$\displaystyle F_p(z)=\sum_{l \ge 1} \beta_p(l) z^l$.
	Therefore, we focus on the function $F_p(z)$.
	Since $1 \ge \alpha(p) \ge \displaystyle \limsup_{ n \to \infty} \beta_p(n)^{1/n}$,
	we have that $F_p(z)$ is finite for all $|z|<1$.
	When $p < p_c(\mathbb{Z}^2)=1/2$, 
	we know that $\displaystyle \limsup_{n \to \infty} \alpha_p(n) =0$ holds.
	In the case of $z=1$, since
	\begin{align*}
	\alpha_p(n)
	&= \alpha_p(n-1) - \left( \frac{1-p}{1-p+p^2} \right)^2 \beta_p(n-1)\\
	&=\alpha_p(1)- \left( \frac{1-p}{1-p+p^2} \right)^2
		\sum_{k=1}^{n-1} \beta_p(k),
	\end{align*}
	we have that 
	\[
	F_p(1)=\left( \frac{1-p+p^2}{1-p} \right)^2 \alpha_p(1).
	\]
	Thus, we would like to consider whether $F_p(z)$ converges or not 
	in $|z| \ge1, z \not=1$.
	We set 
	\begin{align*}
	\Phi_p(l)&= \prod_{i=1}^l \frac{1-p}{1-p+p^{i+1}},\\
	H_p(z)&= \sum_{l \ge 1} \Phi_p(l)^2 z^l.
	\end{align*}
	It is easy to show that the radius of convergence of $H_p(z)$ equals 1.
	But we would like to consider $|z| \ge 1, z\not=1$.
	Therefore, we consider an analytic continuation of $H_p(z)$. 
	We have
	\[
	\Phi_p(l)^2 
	= \left( \frac{1-p+p^{(l+1)+1}}{1-p}\right)^2 \Phi_p(l+1)^2
		= \left( 1+ \frac{2p^2}{1-p}p^l + \frac{p^4}{(1-p)^2}p^{2l} \right)
		\Phi_p(l+1)^2 .
	\]
	Using this equation, $H_p(z)$ is deformed into the form
\begin{align*}
	H_p(z)
	&= \sum_{l \ge 1} \Phi_p(l+1)^2 z^l 
		+ \frac{2p^2}{1-p} \sum_{l \ge 1} \Phi_p(l+1)^2 (pz)^l
		+ \frac{p^4}{(1-p)^2} \sum_{l \ge 1} \Phi_p(l+1)^2 (p^2 z)^l\\
	&= \frac{1}{z}H_p(z) + \frac{2p}{1-p}\frac{1}{z}H_p(pz)
		+ \left( \frac{p}{1-p} \right)^2 \frac{1}{z}H_p(p^2 z)-1.
\end{align*}
	Then we have 
	\begin{equation}
	\label{eq:H}
	H_p(z)
	= \frac{\frac{2p}{1-p} H_p(pz) 
		+\left( \frac{p}{1-p} \right)^2 H_p(p^2 z)- z}{z-1}.
	\end{equation}
	Therefore, the right-hand side is the analytic continuation of $H_p(z)$
	defined in $|z|<1/p, z\not=1$.
	\begin{lem}
	\label{lem:pole}
	If there exists $x_0 \in (1,1/p) \subset \mathbb{R}$ such that $H_p(x_0)=-1$,
	then $x_0$ is a pole of $F_p(z)$. 
	\end{lem}
	Lemma \ref{lem:pole} means that the radius of convergence of $F_p(z)$ is 
	less than or equal to $x_0$.
	Then we have $\alpha(p) \ge 1/x_0$.
	\begin{proof}	
	First, we consider the relationship between $\beta_p$ and $\Phi_p$.
	We will show the following equations.
	\begin{align}
	\label{eq:beta1}
	\beta_p(1)
	&= \frac{p}{1-p} \Phi_p(2)^2(1-p^2(1-p)(1-p^2)),\\
	\label{eq:betan}
	\beta_p(n)
	&= \frac{p}{(1-p)^2}\Phi_p(n+1)^2 \left( 1-p^n)(1-p^2(1-p)(1-p^{n+1}) \right)
	 - \sum_{k=1}^{n-1} \Phi_p(n-k)^2 \beta_p(k).
	\end{align}
	We define $B_p(i)=p^{m_i+l_i}(1-p)^{2+m_i}(m_i+1)\binom{m_i+l_{i-1}}{l_i}$,
	then we can express $\beta_p(n)$ in terms of $B_p(i)$ as
	\[
	\beta_p(n)=(1-p)
	\sum_{\substack{m_i \ge 0 \\ 1 \le i \le n}}
	\sum_{\substack{ l_i =1   \\ 1 \le i \le n}}^{m_i +l_{i-1}}
	\left( \prod_{j=1}^n B_p(j) \right).
	\]
	We set $S_p(t)=\sum_{k=0}^t p^k$.
	Then the following equation holds.
\begin{align*}
	\sum_{m_i \ge 0} \sum_{l_i=1}^{m_i+l_{i-1}} B_p(i) S_p(t)^{l_i}
	&= (1-p)^2\sum_{m_i \ge 0} (m_i+1) p^{m_i} (1-p)^{m_i} 
		\sum_{l_i=1}^{m_i+l_{i-1}} (pS_p(t))^{l_i} \binom{m_i + l_{i-1}}{l_i}\\
	&= (1-p)^2\sum_{m_i \ge 0} (m_i+1)p^{m_i} (1-p)^{m_i} 
		\left( (1+pS_p(t))^{m_i+l_{i-1}}-1 \right)\\
	&= \left( \frac{1-p}{1-p+p^{t+3}}\right)^2 S_p(t+1)^{l_{i-1}} 
		- \left( \frac{1-p}{1-p+p^2}\right)^2.
\end{align*}
	If $i=1, t=0$, then we have 
\begin{align*}
	\beta_p(1)
	&= \left( \frac{1-p}{1-p+p^{3}} \right)^2S_p(1)^1 
		- \left( \frac{1-p}{1-p+p^2} \right)^2\\
	&= \frac{p}{1-p} \Phi_p(2)^2(1-p^2(1-p)(1-p^2)).
\end{align*}
	In the case where $n \ge 2$, we see
\begin{align*}
	\beta_p(n)
	&= (1-p)\sum_{\substack{m_i \ge 0 \\ 1 \le i \le n-1}}
		\sum_{\substack{ l_i =1   \\ 1 \le i \le n-1}}^{m_i +l_{i-1}}
		\left( \prod_{j=1}^{n-1} B_p(j) \right)
		\sum_{m_n \ge 0} \sum_{l_n=1}^{m_n+l_{n-1}} B_p(n) S_p(0)^{l_n}\\
	&=  (1-p)\sum_{\substack{m_i \ge 0 \\ 1 \le i \le n-1}}
		\sum_{\substack{ l_i =1   \\ 1 \le i \le n-1}}^{m_i +l_{i-1}}
		\left( \prod_{j=1}^{n-1} B_p(j) \right)
		\left( \left( \frac{1-p}{1-p+p^3}\right)^2 S_p(1)^{l_{n-1}} 
		- \left( \frac{1-p}{1-p+p^2}\right)^2  \right)\\
	&= (1-p)\left( \frac{1-p}{1-p+p^3}\right)^2
		\sum_{\substack{m_i \ge 0 \\ 1 \le i \le n-2}}
		\sum_{\substack{ l_i =1   \\ 1 \le i \le n-2}}^{m_i +l_{i-1}}
		\left( \prod_{j=1}^{n-2} B_p(j) \right)
		\sum_{m_{n-1} \ge 0} \sum_{l_{n-1}=1}^{m_{n-1}+l_{n-2}} B_p(n-1) S_p(1)^{l_{n-1}}\\
		&\quad - \left( \frac{1-p}{1-p+p^2}\right)^2\beta_p(n-1).
\end{align*}
	Repeating this process, we have
	\begin{align*}
	\beta_p(n)
	&= (1-p)\Phi_p(n)^2  \left( \frac{1-p+p^2}{1-p} \right)^2 
		\left( \left( \frac{1-p}{1-p+p^{(n+1)+1}}\right)^2 S_p(n)^1
		-\left( \frac{1-p}{1-p+p^2}\right)^2 \right)\\
		&\quad - \sum_{k=1}^{n-1} \Phi_p(n-k) \beta_p(k)\\
	&= \frac{p}{(1-p)^2}\Phi_p(n+1)^2 \left( 1-p^n)(1-p^2(1-p)(1-p^{n+1}) \right)
	 - \sum_{k=1}^{n-1} \Phi_p(n-k)^2 \beta_p(k).
	\end{align*}
	By the equations \eqref{eq:beta1} and \eqref{eq:betan},
	we obtain the following expression of $F_p(z)$ in terms of $\Phi_p(l)$.
	\begin{align*}
	F_p(z)
	&= \sum_{l \ge 2}  \left( \frac{p}{(1-p)^2}\Phi_p(l+1)^2(1-p^l)
		(1-p^2(1-p)(1-p^{l+1}))
		- \sum_{k=1}^{l-1} \Phi_p(l-k)^2 \tilde{\beta}_p(k) \right) z^l
		+ \beta_p(1)z\\
	&= \frac{p}{(1-p)^2}(1-p^2+p^3)\sum_{l \ge 1} \Phi_p(l+1)^2 z^l
		-\frac{p}{(1-p)^2}(1-p^2+2p^3-p^4)\sum_{l \ge 1} \Phi_p(l+1)^2 (pz)^l\\
	&\quad +\frac{p}{(1-p)^2}(p^3-p^4)\sum_{l \ge 1} \Phi_p(l+1)^2 (p^2 z)^l
		-\sum_{l \ge 2}\sum_{k=1}^{l-1} \Phi_p(l-k)^2 \beta_p(k)z^l.
	\end{align*}
	Each term in the above expression is rewritten as 
\begin{align*}
	\sum_{l \ge 1} \Phi_p(l+1)^2 z^l
	&= \frac{1}{z} \sum_{l \ge 2} \Phi_p(l)^2 z^l 
		= \frac{1}{z}(H_p(z) - \Phi(1)^2 z),\\
		\sum_{l \ge 1} \Phi_p(l+1)^2 (pz)^l
	&= \frac{1}{pz} \sum_{l \ge 2} \Phi_p(l)^2 (pz)^l 
		= \frac{1}{pz}(H_p(pz) - \Phi(1)^2 pz),
\end{align*}
\begin{align*}
		\sum_{l \ge 1} \Phi_p(l+1)^2 (p^2 z)^l
	&= \frac{1}{p^2 z} \sum_{l \ge 2} \Phi_p(l)^2 (p^2 z)^l 
		= \frac{1}{p^2 z}(H_p(p^2 z) - \Phi(1)^2 p^2 z),\\
		\sum_{l \ge 2}\sum_{k=1}^{l-1} \Phi_p(l-k)^2 \beta_p(k)z^l
	&= \sum_{k \ge 1} \beta_p(k)z^l \sum_{l-k \ge 1} \Phi_p(l-k)^2 z^{l-k}
		= F_p(z) H_p(z).
\end{align*}
	Then we obtain
	\begin{align}
	\label{eq:F}
	(1+H_p(z)) F_p(z) 
	&= \frac{p}{(1-p)^2} \left( \frac{1-p^2(1-p)}{z} H_p(z) 
		- \frac{1-p^2(1-p^2)}{pz}  H_p(pz)  
	    - \frac{p^3(1-p)}{p^2z} H_p(p^2 z) \right) .
	\end{align}
	By the equation \eqref{eq:F},
	it is enough to show the right hand side of \eqref{eq:F} is not equal to 0
	when $H_p(x_0)=-1$.
	By the equation (\ref{eq:H}), 
	we have 
	\[
	H_p(p^2x) =
	\left( \frac{1-p}{p} \right)^2 
	- 2 \frac{1-p}{p} H_p(px).
	\]
	Using this equation, if the right hand side of \eqref{eq:F} is equal to 0,
	then we have $H_p(px_0)=-1$.
	This is contrary to $H_p(px) \ge 0$ for all $x \in [0, 1/p)$.
	\end{proof}
	We set
	\[
	h_p(x) = 2H_p(px) +\frac{p}{1-p} H_p(p^2x ).
	\]
	\begin{lem}
	\label{lem:key}
	Let $p_0=\inf \left\{ p\in [0,1] \mid  \exists x \in(1,1/p) \text{ s.t. }
	 h_p(x) \ge (1-p)/p,
	x < (d-1) \right\}$. Then we have $p_c (T_d \Box \mathbb{Z}) \le p_0$.
	\end{lem}
	\begin{proof}
	For any $p>p_0$,
	there exists $x \in (1,1/p)$ such that $h_p(x) \ge (1-p)/p, x < (d-1)$.
	By the equation (\ref{eq:H}), 
	$H_p(x_0)=-1$ if and only if 
	$h_p(x_0)=(1-p)/p$.
	The function $h_p(x)$ is continuous, increasing on $[0,1/p)$,
	and $h_p(0)=0$.
	Therefore, we have $\alpha(p) \ge 1/x_0 \ge 1/x >1/(d-1)$.
	Using Lemma \ref{lem:limit}, we have $p_c (T_d \Box \mathbb{Z}) \le p_0$.
	\end{proof}

\section{Proof of TheoremA}
\label{sec:thmA}
	As a candidate of the real number $x$ which appeared in Lemma\ref{lem:key},
	we consider $x=(1-p)/p$.
	If $p>1/d$, then $x<(d-1)$ holds.
	Therefore, we must only show $h_p(x) \ge (1-p)/p$.
	Now we let $d \ge 3$ and assume $1/d<p<0.34$.
	Let $x_1, \ldots, x_l >0$ be real numbers.
	By the relation between the harmonic mean and the geometric mean,
	we have 
	\begin{equation}
	\left( \prod_{k=1}^l x_k \right)^{\frac{1}{l}} 
	\ge \left( \frac{1}{l} \sum_{k=1}^l \frac{1}{x_k} \right)^{-1}.
	\end{equation}
	Using this inequality, we have
\begin{align*}
	\Phi_p(l)
	&= \prod_{k=1}^l \frac{1-p}{1-p+p^{k+1}}\\
	&\ge \left( \frac{1}{l}\sum_{k=1}^l \frac{1-p+p^{k+1}}{1-p} \right)^{-l}
	=\left( \frac{l}{l+\left(\frac{p}{1-p} \right)^2 (1-p^l)} \right)^l\\
	&\ge \left( \frac{l}{l+\left(\frac{p}{1-p} \right)^2 } \right)^l
	=\left( 1- \frac{ \left( \frac{p}{1-p}\right)^2}
		{l+ \left( \frac{p}{1-p}\right)^2} \right)^l\\
	&\ge \left( 1- \frac{ \left( \frac{p}{1-p}\right)^2}{l}\right)^l
	\ge e^{-\left( \frac{p}{1-p}\right)^2}.
\end{align*}
	From this inequality, we obtain
\begin{align*}
	h_p\left( \frac{1-p}{p} \right)
	&\ge 2H_p(1-p)
	= 2\sum_{l \ge 1} \Phi_p(l)^2 (1-p)^l \\
	&\ge 2  e^{-2 \left( \frac{p}{1-p}\right)^2} \sum_{l \ge 1} (1-p)^l\\
	&=2 e^{-2 \left( \frac{p}{1-p}\right)^2} \frac{1-p}{p}.
\end{align*}
	Since $p$ is assumed to satisfy $p < 0.34$,
	then we have 
	\[
	2 e^{-2 \left( \frac{p}{1-p}\right)^2} \ge 1.
	\]
	Hence $h_p(x_0) \ge (1-p)/p$ holds.

\section{Proof of TheoremB}
\label{sec:thmB}
	In Section\ref{sec:thmA},
	we have
	\[
	\Phi_p(l)^2 \ge e^{-2 \left( \frac{p}{1-p}\right)^2}
	\ge 1-2 \left(\frac{p}{1-p} \right)^2.
	\]
	Using this inequalty,
	we have 
	\begin{align*}
	h_p(x) &\ge 
	2 \left( 1-2\left(\frac{p}{1-p} \right)^2 \right) \sum_{l \ge 1} (px)^l
	+\frac{p}{1-p} \left( 1-2\left(\frac{p}{1-p} \right)^2 \right) \sum_{l \ge 1} (p^2x)^l\\
	&= 2 \left( 1-2\left(\frac{p}{1-p} \right)^2 \right) \frac{px}{1-px}
	+\frac{p}{1-p} \left( 1-2\left(\frac{p}{1-p} \right)^2 \right) \frac{p^2x}{1-p^2x}.
	\end{align*}
	In the case of $d=4$, let $p=0.225, x=2.999$.
	Then we have $h_p(x)\ge (1-p)/p, x<d-1$.
	Therefore, $p_c<0.225$ holds.
	Using the inequality \eqref{ineq:p_u}, 
	we have $p_u > 0.232$.
	Hence, $p_c <p_u$ holds when $d=4$.
{}

\noindent
{Mathematical Institute \\
 Tohoku University \\
 Sendai 980-8578 \\
 Japan}
\end{document}